\documentclass[11pt,letterpaper,reqno]{amsart}
\newcommand{\ppar}[1]{\par\addvspace{\medskipamount}\noindent{\bfseries #1.}\enspace\ignorespaces}
\usepackage{amsmath,amssymb,amsthm,amsfonts,mathtools}
\usepackage{mathrsfs}
\usepackage{xcolor}
\usepackage[T1]{fontenc}
\usepackage{microtype}
\usepackage{aliascnt}

\usepackage[colorlinks=true,
  linkcolor=blue,
  citecolor=blue,
  urlcolor=blue]{hyperref}
\usepackage[nameinlink,capitalize,noabbrev]{cleveref}

\newtheorem{theorem}{Theorem}[section]

\newtheorem{corollary}[theorem]{Corollary}

\newtheorem{conjecture}[theorem]{Conjecture}
\theoremstyle{definition}

\crefname{theoremletter}{Theorem}{Theorems}
\Crefname{theoremletter}{Theorem}{Theorems}

\newcommand{\R}{\mathbb R}
\newcommand{\tr}{\operatorname{tr}}
\newcommand{\ip}[2]{\left\langle #1,#2\right\rangle}
\newcommand{\norm}[1]{\left\lVert #1\right\rVert}
\newcommand{\ones}{\mathbf 1}

\begin{document}

\title{The positive and negative square-energy conjecture}

\author[Y.~Liu]{Yinchen Liu}
\address{Institute for Interdisciplinary Information Sciences, Tsinghua University, Beijing 100084, P. R. China}
\email{liuyinch23@mails.tsinghua.edu.cn}

\author[Q.~Tang]{Quanyu Tang}
\address{School of Mathematical Sciences, University of Science and Technology of China, Hefei 230026, P. R. China}
\email{tangquanyu827@gmail.com}

\author[S.~Zhang]{Shengtong Zhang}
\address{Department of Mathematics, Stanford University, Stanford, CA 94305, USA}
\email{stzh1555@stanford.edu}

\subjclass[2020]{Primary 05C50; Secondary 15A18, 15A42}

\keywords{positive square energy, negative square energy,
doubly nonnegative matrix, semidefinite programming}

\begin{abstract}
Let $s^+(G)$ and $s^-(G)$ denote the sums of the squares of the positive and negative adjacency eigenvalues of a graph $G$, respectively. We prove the conjecture of Elphick, Farber, Goldberg, and Wocjan that every connected graph $G$ on $n$ vertices satisfies
$$
    \min\{s^+(G),s^-(G)\}\ge n-1.
$$
The proof introduces a new framework for square-energy estimates, in which the Hadamard squares of positive semidefinite matrices that encode these spectral quantities are relaxed to the full doubly nonnegative cone.
\end{abstract}

\maketitle

\section{Introduction}

All graphs in this paper are finite, simple, and undirected.  Let
$G=(V,E)$ be a graph with $n=|V|$ vertices, $m=|E|$ edges, and adjacency
matrix $A=A(G)$, and let $\lambda_1,\ldots,\lambda_n$ be the eigenvalues
of $A$, counted with multiplicity.  The \emph{positive square energy}
and \emph{negative square energy} of $G$ are defined by
\[
    s^+(G):=\sum_{\lambda_i>0}\lambda_i^2,
    \qquad
    s^-(G):=\sum_{\lambda_i<0}\lambda_i^2.
\]
These invariants are quadratic analogues of the classical graph
energy~\cite{LSG}, and they arise naturally in spectral bounds on the
chromatic number: confirming a conjecture of Wocjan and
Elphick~\cite{WE}, Ando and Lin~\cite{AL} proved that every graph
with at least one edge satisfies
\[
    \chi(G)\ge1+\max\biggl\{\frac{s^+(G)}{s^-(G)},
    \frac{s^-(G)}{s^+(G)}\biggr\}.
\]

Here and below, $X\succeq0$ means that $X$ is positive semidefinite.
If
\[
    A=A_+-A_-,
    \qquad
    A_+,A_-\succeq0,
    \qquad
    A_+A_-=0
\]
is the decomposition of $A$ into its positive and negative spectral
parts, then
\[
    s^+(G)=\tr(A_+^2),
    \qquad
    s^-(G)=\tr(A_-^2).
\]
Since $\tr(A^2)=2m$, we have
\begin{equation}\label{eq:trace}
    s^+(G)+s^-(G)=2m.
\end{equation}

In 2016, Elphick, Farber, Goldberg, and
Wocjan~\cite[Conjecture~1]{EFGW} proposed the following conjecture.

\begin{conjecture}\label{conj:EFGW}
Every connected graph $G$ on $n$ vertices satisfies
\[
\min\{s^+(G),s^-(G)\}\ge n-1.
\]
\end{conjecture}

This conjecture was later highlighted as Conjecture~1 in the survey
of Liu and Ning~\cite{LiuNing2023}.

By~\eqref{eq:trace}, the conjecture states, equivalently, that every
connected graph satisfies $\max\{s^+,s^-\}\le 2m-n+1$, strengthening
Hong's classical spectral radius bound
$\lambda_1^2\le 2m-n+1$~\cite{Hong}.  The conjectured bound is sharp:
every tree satisfies $s^+(G)=s^-(G)=m=n-1$, and the complete graph
satisfies $s^-(K_n)=n-1$.

Since its formulation, Conjecture~\ref{conj:EFGW} has attracted
considerable attention.  The original paper~\cite{EFGW} established it
for several classes of graphs, including regular and bipartite graphs,
and Abiad, de~Lima, Desai, Guo, Hogben, and Madrid~\cite{AbiadEtAl}
proved it for several additional classes.  Further work has examined
structural properties and the asymmetry between the two square
energies~\cite{EL}, as well as a finer conjecture for unicyclic
graphs~\cite{NZ}.  Substantial progress was also made on the general
case: Zhang~\cite{Zh} established
\[
    \min\{s^+(G),s^-(G)\}\ge n-\gamma(G),
\]
where $\gamma(G)$ denotes the domination number of $G$---in
particular, $\min\{s^+,s^-\}\ge n/2$---while Akbari, Kumar, Mohar,
and Pragada~\cite{AKMP} simultaneously proved the linear
bound
\[
    \min\{s^+(G),s^-(G)\}\ge \frac{3n}{4};
\]
see
also~\cite{AKMPZ} for refinements concerning the positive square
energy, and~\cite{CS,CSZ} for a systematic study of the semidefinite
and conic methods behind these bounds.

Our main result proves Conjecture~\ref{conj:EFGW} in full.

\begin{theorem}\label{thm:main}
Every connected graph $G$ on $n$ vertices satisfies
\[
    \min\{s^+(G),s^-(G)\}\ge n-1.
\]
\end{theorem}

Our proof is based on a general relaxation: since the Hadamard
squares $A_\pm\circ A_\pm$ are doubly nonnegative, bounds on the
square energies follow from inequalities valid over the entire doubly
nonnegative cone.  At the heart of this paper is one such inequality,
for matrices indexed by the vertices of a connected graph
(Theorem~\ref{thm:sparse}), which we believe to be of independent
interest.  

Beyond Conjecture~\ref{conj:EFGW} itself,
Theorem~\ref{thm:main} resolves several related conjectures from the
literature.  More precisely, it proves the disconnected version of the
square-energy conjecture~\cite[Conjecture~2]{EFGW}, Elphick's
conjecture on adjacency inertia recorded in
\cite[Conjecture~7.4]{Zh}, and conjectured lower bounds for positive
and negative $p$-energies from
\cite[Conjecture~2]{AKMPp} and \cite[Conjecture~5.4]{ETZ}.
It also establishes unconditionally a Nordhaus--Gaddum-type inequality
that Elphick and Aouchiche~\cite[Theorem~4]{EA} derived under
Conjecture~\ref{conj:EFGW}.  These consequences are presented in
Section~\ref{sec:consequences}.  In the companion
paper~\cite{LTZTuran}, the same relaxation yields a positive
square-energy strengthening of Tur\'an's theorem; the two proofs are
otherwise independent.

\subsection{Notation}
For real symmetric matrices $X$ and $Y$, write
\[
    \ip XY:=\tr(XY),
    \qquad
    \norm X_F:=\sqrt{\ip XX}.
\]
We use $X\ge0$ for entrywise nonnegativity; a real symmetric matrix
$X$ is \emph{doubly nonnegative} if $X\succeq0$ and $X\ge0$ (see
\cite{BSM} for a textbook treatment).

Let $\ones$ be the all-ones
vector, set $J:=\ones\ones^{\top}$, and write $X\circ Y$ for the
Hadamard product of $X$ and $Y$.  Whenever a graph $G$ is fixed and
$u\in V(G)$, let $e_u$ denote the corresponding standard basis vector;
sums of the form $\sum_{uv\in E(G)}$ run over the edges of $G$, each
edge counted once.

\subsection{Origin of the proof}
The proof given below is short, but its two ingredients---Theorem~\ref{thm:sparse} and the reduction in Section~\ref{sec:main}---would be difficult to motivate in isolation. In this subsection, we give an informal explanation of where
these ingredients come from. Readers interested only in the formal proof
may skip directly to Section~\ref{sec:sparse}.

Our starting point is an observation underlying the semidefinite
method of~\cite{AL,CS,CSZ,Zh}: the square energies are themselves
optimal values of semidefinite programs.  Indeed, for every
$X\succeq0$,
\[
    \ip AX=\ip{A_+}{X}-\ip{A_-}{X}
    \le\ip{A_+}{X}
    \le\norm{A_+}_F\norm X_F
    =\sqrt{s^+(G)}\,\norm X_F,
\]
with equality throughout at $X=A_+$, and therefore
\begin{equation}\label{eq:variational}
    s^+(G)=\max_{0\ne X\succeq0}
    \frac{\max\{\ip AX,\,0\}^2}{\norm X_F^2},
    \qquad
    s^-(G)=\max_{0\ne X\succeq0}
    \frac{\max\{-\ip AX,\,0\}^2}{\norm X_F^2}
\end{equation}
(equivalently, $s^\pm(G)=\min_{M\succeq0}\norm{A\pm M}_F^2$, the form
emphasized in~\cite{Zh}).  Consequently, an upper bound on both square
energies is the same thing as a \emph{certificate}: a constant $\beta$
with
\[
    \ip{A}{X}^{2}\le\beta\,\norm X_F^2
    \qquad\text{for every }X\succeq0.
\]

\ppar{Earlier certificates}
Ando and Lin~\cite{AL}, proving a conjecture of Wocjan and
Elphick~\cite{WE}, obtained the certificate
$\beta=2m\bigl(1-\tfrac1{\chi(G)}\bigr)$, and Coutinho and
Spier~\cite{CS} sharpened it to
$\beta=2m\bigl(1-\tfrac1{\chi_v(G)}\bigr)$, where
$\chi_v(G)\le\chi(G)$ is the vector chromatic number.  Both arguments
use the same mechanism, with a parameter $t\in\{\chi(G),\chi_v(G)\}$.
Applying the Cauchy--Schwarz inequality to the $2m$ nonzero entries of
the adjacency matrix gives
\[
    \ip{A}{X}^{2}
    =
    \Bigl(\sum_{u,v}A_{uv}X_{uv}\Bigr)^2
    \le
    2m\sum_{u,v}A_{uv}X_{uv}^{2}
    =
    2m\,\ip{A}{X\circ X},
\]
and the parameter $t$ enters through an inequality on the doubly
nonnegative cone: the matrix $Z:=X\circ X$ satisfies $Z\ge0$, and also
$Z\succeq0$ by the Schur product theorem~\cite[\S7.5]{HJ}, and every such
doubly nonnegative matrix obeys
\begin{equation}\label{eq:dnn}
    \ip{A}{Z}
    \le
    \left(1-\frac{1}{t}\right)\ip{J}{Z}.
\end{equation}
Since $\ip{J}{X\circ X}=\norm{X}_F^2$, the two inequalities combine
with~\eqref{eq:variational} to give
$s^\pm(G)\le2m\bigl(1-\tfrac1t\bigr)$.  The conic optimization problem
underlying~\eqref{eq:dnn} is studied systematically by Coutinho,
Spier, and Zhang~\cite{CSZ}.

\ppar{The idea behind \eqref{eq:dnn}}
Following the original proofs in~\cite{AL,CS}, the key to~\eqref{eq:dnn}
is the construction of a correlation matrix $C$ (i.e., $C\succeq0$ with
$C_{vv}=1$ for all $v\in V$) such that $C_{uv}\le-\frac1{t-1}$ for
every edge $uv\in E$.  For the parameter $t=\chi(G)$, such a $C$ comes
from a proper coloring made geometric: let
$y_1,y_2,\ldots,y_t\in\R^{t-1}$ be the unit vectors of a regular
simplex centered at the origin, so that
$\ip{y_i}{y_j}=-\frac1{t-1}$ for all $i\ne j$, and set
$c_u:=y_{\text{color of }u}$.  The Gram matrix
$C_{uv}:=\ip{c_u}{c_v}$ is then a correlation matrix with
$C_{uv}=-\frac1{t-1}$ for all adjacent $u$ and $v$, because adjacent
vertices receive distinct colors.  (Relaxing proper colorings to
arbitrary unit vectors with $\ip{c_u}{c_v}\le-\frac1{t-1}$ on edges---a \emph{vector coloring}---is what sharpens the parameter to
$t=\chi_v(G)$~\cite{CS}.)

Under this construction $(J-C)\circ A\ge\tfrac t{t-1}A$, because
$A\ge0$ and $1-C_{uv}\ge\tfrac t{t-1}$ on edges; hence
$\ip{(J-C)\circ A}{Z}\ge\tfrac t{t-1}\ip AZ$ for every $Z\ge0$.  The
claim \eqref{eq:dnn} is then an instance of an identity valid for
\emph{every} correlation matrix $C$:
\begin{equation}\label{eq:conic}
    J-(J-C)\circ A
    \;=\;
    \underbrace{\vphantom{(J-C)}C}_{\succeq\,0}
    \;+\;
    \underbrace{(J-C)\circ(J-A)}_{\ge\,0},
\end{equation}
where $J-C\ge0$ because $C\succeq0$ with unit diagonal: every
$2\times2$ principal minor of $C$ gives $|C_{uv}|\le1$.  The
right-hand side is a sum of a positive semidefinite and a nonnegative
matrix, so pairing with any doubly nonnegative $Z$ yields
\[
    \frac t{t-1}\ip AZ\le\ip{(J-C)\circ A}{Z}\le\ip JZ,
\]
the first inequality being the estimate just noted; this is exactly
\eqref{eq:dnn}.

\ppar{A weighted estimate}
In the language of correlation matrices, the two estimates above
collapse into a single Cauchy--Schwarz step, with weight $1-C_{uv}$ on
the edge $uv$:\footnote{For the coloring matrix above,
$1-C_{uv}=\tfrac t{t-1}$ on every edge, so the weights are uniform;
a general correlation matrix may price each edge differently.}
\[
    \ip{A}{X}^{2}
    =\Bigl(2\sum_{uv\in E}X_{uv}\Bigr)^2
    \le\Phi_G(C)\cdot\ip{(J-C)\circ A}{X\circ X}
    \le\Phi_G(C)\,\norm X_F^2,
\]
where the last inequality uses~\eqref{eq:conic}. For a correlation matrix
$C$, define
\[
    \Phi_G(C):=
    \begin{cases}
        \displaystyle
        \sum_{uv\in E(G)}\frac{2}{1-C_{uv}},
        & C_{uv}<1\text{ for every }uv\in E(G),\\[3mm]
        +\infty,
        & C_{uv}=1\text{ for some }uv\in E(G).
    \end{cases}
\]
Here $C_{uv}\le 1$ for all $u,v$, since $C$ is a correlation matrix.
Every correlation matrix $C$ with $\Phi_G(C)<\infty$ therefore gives a
certificate of value $\Phi_G(C)$, and it is natural to minimize this
quantity. Theorem~\ref{thm:main} would follow from the bound
\[
    \min_{C}\Phi_G(C)\le q(G):=2m-n+1
    \qquad\text{for every connected graph }G.
\]

\ppar{Bounding the optimum by duality}
Minimizing $\Phi_G$ over correlation matrices is a convex problem
that can be cast as a semidefinite program: the variable $C$ ranges
over the convex set
\[
    \mathrm{Corr}_n:=\{C\in\R^{n\times n}:
    C=C^{\top},\ C\succeq0,\ C_{vv}=1\text{ for all }v\},
\]
cut out by linear equations and one semidefinite constraint.  The
well-known first-order conditions for semidefinite programs (see
\cite[\S3]{VB}) reveal what a minimizer $C$ must look like:
complementary slackness produces a doubly nonnegative matrix $R$,
supported off-diagonal on $E(G)$, with
\[
    R_{uv}=\frac1{(1-C_{uv})^2}\quad(uv\in E),
    \qquad
    \sum_{uv\in E}\sqrt{R_{uv}}=\frac{\Phi_G(C)}2,
    \qquad
    \ones^{\top}R\ones=\Phi_G(C).
\]
The two normalization identities suggest forgetting where $R$ came
from and, guided by homogeneity, guessing that
\begin{equation}\label{eq:thm2.1}
        4\Bigl(\sum_{uv\in E}\sqrt{M_{uv}}\Bigr)^2
    \le q(G)\cdot\ones^{\top}M\ones
\end{equation}
for every connected graph $G$ and every doubly nonnegative matrix $M$
supported on the edges.  Taking $M=R$ then gives
$\Phi_G(C)^2\le q(G)\,\Phi_G(C)$, i.e.\ $\Phi_G(C)\le q(G)$, as
desired.  Once formulated, the inequality admits a short inductive
proof: it follows by induction on cut vertices---both $q$ and
Gram representations split there---with the $2$-connected case handled
by averaging the induction hypothesis over all one-vertex deletions.

\ppar{A direct reduction}
In the final write-up the semidefinite program disappears altogether.
A folding trick disposes of the support hypothesis, so that \eqref{eq:thm2.1} (Theorem~\ref{thm:sparse}) holds for all doubly nonnegative matrices,
and applying it to $M=A_\pm\circ A_\pm$ yields Theorem~\ref{thm:main}
directly.  In light of the variational
characterization~\eqref{eq:variational}, this reduction says that
Theorem~\ref{thm:sparse} relaxes Theorem~\ref{thm:main} from the
Hadamard squares $\{X\circ X:X\succeq0\}$, which compute the square
energies exactly, to the full doubly nonnegative cone. The
correlation-matrix program survives only in this overview, as the
account of how the inequality was found.

\subsection{Organization}
The rest of
the paper is organized accordingly: Section~\ref{sec:sparse} proves
Theorem~\ref{thm:sparse}, Section~\ref{sec:main} deduces
Theorem~\ref{thm:main}, and Section~\ref{sec:consequences} collects
consequences of the main theorem.

\ppar{Formal verification}
Theorem~\ref{thm:main} has been formally
verified in Lean~4 using Mathlib.  The complete source code,
documentation, reproducible build instructions, and kernel-axiom audit
are available at
\href{https://github.com/ShengtongZhang-alt/Sq}
{\texttt{github.com/ShengtongZhang-alt/Sq}}.

\section{A doubly nonnegative matrix inequality}\label{sec:sparse}
For a connected graph $G$, set
\[
    q(G):=2|E(G)|-|V(G)|+1;
\]
by~\eqref{eq:trace}, Theorem~\ref{thm:main} is equivalent to the
upper bound $\max\{s^+(G),s^-(G)\}\le q(G)$, which we will deduce in
Section~\ref{sec:main} from the following inequality.

\begin{theorem}\label{thm:sparse}
Let $G$ be a simple connected graph and let $M$ be a doubly nonnegative matrix
indexed by $V(G)$. Then
\[
    4\cdot \Bigl(\sum_{uv\in E(G)}\sqrt{M_{uv}}\Bigr)^{\!2}
    \le q(G)\cdot\ones^{\top}M\ones.
\]
\end{theorem}
\begin{proof}
For any pair $(G,M)$ as in the theorem, write
\[
    S(G,M):=\sum_{uv\in E(G)}\sqrt{M_{uv}},
    \qquad
    T(M):=\ones^{\top}M\ones,
\]
so that the claim reads $4S(G,M)^2\le q(G)\,T(M)$.  We induct on
$n:=|V(G)|$.  If $n=1$, the result follows from $\ones^{\top}M\ones \ge 0$. If $n=2$, then $G=K_2$,
$q(K_2)=1$, and $M=\begin{psmallmatrix}a&c\\c&b\end{psmallmatrix}$ with
$c\ge0$; positive semidefiniteness gives
$c\le\sqrt{ab}\le\tfrac{a+b}2$, so
$4S(G,M)^2=4c\le a+b+2c=T(M)$.  Now assume $n\ge3$ and that the theorem holds for all connected graphs on at most $n-1$ vertices.

We proceed in three steps: first we fold all non-edge entries of $M$ onto the diagonal, so that we may assume $M_{uv}=0$ whenever $u\ne v$ and $uv\notin E(G)$; then we treat the case where $G$ has a cut vertex by splitting $G$ into two smaller induced subgraphs; finally we handle the case where $G-v$ is connected for every vertex $v$ by averaging the induction hypothesis over all vertex deletions.

\medskip
\noindent\textbf{Folding non-edges.}
Set
\[
    N:=M+\sum_{\{u,v\}\in \binom{V(G)}{2}\setminus E(G)}
    M_{uv}\,(e_u-e_v)(e_u-e_v)^{\top},
\]
where each unordered non-edge \(\{u,v\}\) is included exactly once. Each added term is positive semidefinite, cancels the corresponding non-edge entry of
$M$, and moves its mass to the diagonal; hence $N\succeq0$, $N\ge0$,
$N_{uv}=0$ for non-edges $u\ne v$, and $N_{uv}=M_{uv}$ on edges, so
the edge sums of $M$ and $N$ agree.  Moreover
$\ones^{\top}(e_u-e_v)(e_u-e_v)^{\top}\ones=(1-1)^2=0$, so
$\ones^{\top}N\ones=\ones^{\top}M\ones$. Therefore, without loss of generality, we assume $M_{uv} = 0$ for non-edges $u\ne v$.

\medskip
\noindent\textbf{Cut-vertex case.}
Suppose $G-v$ is disconnected for some $v$.  Group its connected components into
two nonempty unions $U,W$ and set
\[
    G_1:=G[U\cup\{v\}],
    \qquad
    G_2:=G[W\cup\{v\}].
\]
Both are connected (each component of $G-v$ sends an edge to $v$), they
share only $v$, their edge sets partition $E(G)$, and counting vertices
and edges gives
\[
    q(G_1)+q(G_2)=q(G).
\]

Since $M\succeq0$, it admits a Gram representation: there are vectors
$z_u$, $u\in V(G)$, in a Euclidean space, with norm $\norm{\cdot}$,
such that $M_{uv}=\ip{z_u}{z_v}$ for all $u,v\in V(G)$.  No edge joins
$U$ and $W$, so the support hypothesis makes the spans
$\mathcal U:=\operatorname{span}\{z_u:u\in U\}$ and
$\mathcal W:=\operatorname{span}\{z_w:w\in W\}$ orthogonal.  Consider the orthogonal decomposition of $z_v$
\[
    z_v=p_U+p_W+p_0,
    \qquad
    p_U\in\mathcal U,\quad
    p_W\in\mathcal W,\quad
    p_0\perp\mathcal U+\mathcal W.
\]
Let $M_1$ be the Gram matrix of $(z_u)_{u\in U}$ together with $p_U+p_0$
at $v$, and $M_2$ that of $(z_w)_{w\in W}$ together with $p_W$ at $v$.
Every off-diagonal entry of $M_1$ or $M_2$ involving $v$ agrees with
the corresponding entry of $M$: for $u\in U$,
$\ip{z_u}{p_U+p_0}=\ip{z_u}{z_v}$ because $z_u\perp p_W,p_0$, and
likewise on the other side.  As all their other entries are entries of
$M$ or squared norms, $M_1$ and $M_2$ are entrywise nonnegative; being
Gram matrices, they are also positive semidefinite.  Thus both pairs
$(G_1,M_1)$ and $(G_2,M_2)$ satisfy the hypotheses of the theorem, and
the edge sums split,
\[
    S(G,M)=S(G_1,M_1)+S(G_2,M_2).
\]
Moreover $T(M)=\bigl\lVert\sum_{u\in V(G)}z_u\bigr\rVert^2$ splits into
the orthogonal parts $\sum_{u\in U}z_u+p_U$, $\sum_{w\in W}z_w+p_W$,
and $p_0$, so
\[
    T(M)=T(M_1)+T(M_2).
\]
By induction, $2S(G_i,M_i)\le\sqrt{q(G_i)\,T(M_i)}$, so the scalar
Cauchy--Schwarz inequality gives
\[
    2S(G,M)\le\sqrt{q(G_1)T(M_1)}+\sqrt{q(G_2)T(M_2)}
    \le\sqrt{q(G)\,T(M)}.
\]

\medskip
\noindent\textbf{No-cut-vertex case.}
Now $G-v$ is connected for every $v$, so the induction hypothesis can
be averaged over all vertex deletions.  Write $m:=|E(G)|$, $q:=q(G)$,
and split $T(M)$ into diagonal and edge mass,
\[
    T(M)=d_0+2w,
    \qquad
    d_0:=\tr M,
    \quad
    w:=\sum_{uv\in E(G)}M_{uv}.
\]
Deleting a vertex $v$ removes from $S(G,M)$ the terms of the edges at
$v$:
\[
    S(G-v,M-v)=S(G,M)-\sigma_v,
    \qquad
    \sigma_v:=\sum_{u\sim v}\sqrt{M_{uv}},
\]
where $M-v$ denotes the submatrix of $M$ obtained by deleting the row and column corresponding to $v$.  Also, $G-v$ has $n-1$
vertices and $m-d(v)$ edges, $d(v)$ being the degree of $v$, so
$q(G-v)=q+1-2d(v)$.  The induction hypothesis for the pair
$(G-v,M-v)$ therefore reads
\[
    2\bigl(S(G,M)-\sigma_v\bigr)\le\sqrt{q(G-v)\,T(M-v)}.
\]
Sum over all $v$: every edge has two endpoints, so
$\sum_v\sigma_v=2S(G,M)$, and Cauchy--Schwarz gives
\[
    2(n-2)\,S(G,M)
    \le\sum_v\sqrt{q(G-v)\,T(M-v)}
    \le\sqrt{\Bigl(\sum_v q(G-v)\Bigr)\Bigl(\sum_v T(M-v)\Bigr)}.
\]
The two sums are elementary counts.  Using $\sum_v d(v)=2m$ and
$4m=2q+2n-2$, we have
\[
    \sum_v q(G-v)= \sum_v (q + 1 - 2d(v)) = n(q+1)-4m=(n-2)(q-1).
\]
Since each diagonal entry of $M$ survives in $n-1$ of the
submatrices $M-v$ while each edge entry survives in $n-2$ of them, we also have
\[
    \sum_v T(M-v)=(n-1)d_0+2(n-2)w=(n-2)\,T(M)+d_0.
\]
Combining the last three displays and dividing by $(n-2)^2$ yields the
\emph{averaged estimate}
\begin{equation}\label{eq:avg}
    4S(G,M)^2\le(q-1)\,T(M)+\frac{q-1}{n-2}\,d_0.
\end{equation}
Cauchy--Schwarz over the $m$ edges gives a second, \emph{flat},
estimate:
\begin{equation}\label{eq:edgeCS}
    4S(G,M)^2\le4m\sum_{uv\in E(G)}M_{uv}=4mw.
\end{equation}
Now we consider the following two cases, depending on the relative sizes of $w$ and $d_0$.
\ppar{Case 1} If $2(n-1)w\le qd_0$, then $4m=2q+2(n-1)$ turns \eqref{eq:edgeCS} into
\[
    4S(G,M)^2\le4mw=2qw+2(n-1)w\le2qw+qd_0=q\,T(M).
\]
\ppar{Case 2} If instead $2(n-1)w>qd_0$, we use the averaged
estimate \eqref{eq:avg}, for which it suffices to show that
$\frac{q-1}{n-2}\,d_0\le T(M)$.  Set $\beta:=m-n+1$, which is
nonnegative because $G$ is connected.  Since
$q-1=(n-2)+2\beta$ and $T(M)=d_0+2w$, this reduces to
\[
    \beta\,d_0\le(n-2)\,w.
\]
Since $G$ is simple, $2\beta=2m-2n+2\le(n-1)(n-2)$, so
$q(n-2)-2\beta(n-1)=(n-1)(n-2)-2\beta\ge0$.  Multiplying the case
hypothesis $qd_0<2(n-1)w$ by $n-2$,
\[
    2\beta(n-1)\,d_0\le q(n-2)\,d_0<2(n-1)(n-2)\,w,
\]
and dividing by $2(n-1)$ gives $\beta d_0<(n-2)w$, as needed.
Then \eqref{eq:avg} yields
\[
    4S(G,M)^2\le(q-1)\,T(M)+T(M)=q\,T(M),
\]
completing the induction.
\end{proof}

\section{Proof of the main theorem}\label{sec:main}
In this section, we directly deduce Theorem~\ref{thm:main} from Theorem~\ref{thm:sparse}. 

\begin{proof}[Proof of Theorem~\ref{thm:main}]
The case $n=1$ is trivial, so let $n\ge2$; then $G$ has an edge, and
since $\tr A=0$, the matrix $A$ has eigenvalues of both signs, so
$s^+(G)>0$ and $s^-(G)>0$.

Take $M:=A_+\circ A_+$, which satisfies $M\ge0$ and, by the Schur
product theorem, $M\succeq0$.  Its entries are $M_{uv}=(A_+)_{uv}^2$,
so, using $A_+A_-=0$,
\[
    \sum_{uv\in E(G)}\sqrt{M_{uv}}
    =\sum_{uv\in E(G)}\bigl|(A_+)_{uv}\bigr|
    \ge\sum_{uv\in E(G)}(A_+)_{uv}
    =\tfrac12\ip A{A_+}
    =\tfrac12\tr(A_+^2)
    =\tfrac12\,s^+(G),
\]
while
\[
    \ones^{\top}M\ones=\sum_{u,v}(A_+)_{uv}^2=\norm{A_+}_F^2
    =s^+(G).
\]
Theorem~\ref{thm:sparse} therefore gives
\[
    s^+(G)^2
    \le4\Bigl(\sum_{uv\in E(G)}\sqrt{M_{uv}}\Bigr)^{\!2}
    \le q(G)\,s^+(G),
\]
and dividing by $s^+(G)>0$ yields $s^+(G)\le q(G)$.

For the negative part, take $M:=A_-\circ A_-$. Now
$\ip A{A_-}=-\tr(A_-^2)=-s^-(G)$, so, using $|x|\ge-x$,
\[
    \sum_{uv\in E(G)}\sqrt{M_{uv}}
    =\sum_{uv\in E(G)}\bigl|(A_-)_{uv}\bigr|
    \ge-\sum_{uv\in E(G)}(A_-)_{uv}
    =\tfrac12\,s^-(G),
\]
and $\ones^{\top}M\ones=\norm{A_-}_F^2=s^-(G)$ as before, whence
$s^-(G)\le q(G)$.  By~\eqref{eq:trace},
\[
    s^\pm(G)=2m-s^\mp(G)\ge2m-q(G)=n-1.
    \qedhere
\]
\end{proof}

\section{Consequences of the main theorem}\label{sec:consequences}

We record several consequences of Theorem~\ref{thm:main}.  Throughout this
section, $\kappa(G)$ denotes the number of connected components of $G$, and
\[
    c(G):=|E(G)|-|V(G)|+\kappa(G)
\]
denotes its cyclomatic number.

\subsection{Disconnected graphs}
The following corollary proves the disconnected version of the square-energy
conjecture~\cite[Conjecture~2]{EFGW}.

\begin{corollary}\label{cor:components}
Let $G$ be a graph on $n$ vertices with $\kappa(G)$ connected components.
Then
\[
    \min\{s^+(G),s^-(G)\}\ge n-\kappa(G).
\]
\end{corollary}

\begin{proof}
Let $G_1,\ldots,G_{\kappa(G)}$ be the connected components of $G$, and set
$n_i:=|V(G_i)|$.  Since the adjacency matrix of $G$ is block diagonal,
$s^+$ and $s^-$ are additive over connected components.  Therefore
Theorem~\ref{thm:main} gives, for $\sigma\in\{+,-\}$,
\[
    s^\sigma(G)
    =\sum_{i=1}^{\kappa(G)}s^\sigma(G_i)
    \ge\sum_{i=1}^{\kappa(G)}(n_i-1)
    =n-\kappa(G).
\qedhere\]
\end{proof}

\subsection{Asymmetry of the square energies}

Elphick and Linz~\cite{EL} call $s^+(G)-s^-(G)$ the \emph{squared spread} of
$G$.  Combining Corollary~\ref{cor:components} with~\eqref{eq:trace}
gives, for either sign $\sigma$,
\[
    n-\kappa(G)\le s^\sigma(G)\le 2m-n+\kappa(G),
    \qquad\text{that is,}\qquad
    |s^\sigma(G)-m|\le c(G).
\]
It follows that
\begin{equation}\label{eq:square-spread-ineq2}
    |s^+(G)-s^-(G)|\le 2c(G).
\end{equation}
The next result strengthens and makes unconditional the upper bound in
\cite[Proposition~4.1]{EL}.

\begin{corollary}\label{cor:absolute-spread}
Every graph $G$ on $n$ vertices satisfies
\[
    |s^+(G)-s^-(G)|\le (n-1)(n-2).
\]
For $n\ge3$, equality holds if and only if $G\cong K_n$.
\end{corollary}

\begin{proof}
Let $\kappa:=\kappa(G)$.  If the connected components of $G$ have orders
$n_1,\ldots,n_\kappa$, then
\[
    |E(G)|
    \le \sum_{i=1}^{\kappa}\binom{n_i}{2}
    \le \binom{n-\kappa+1}{2}.
\]
The second inequality is attained only when one component has
$n-\kappa+1$ vertices and all the remaining components are isolated
vertices.  It follows that
\[
    c(G)
    =|E(G)|-n+\kappa
    \le \binom{n-\kappa}{2}
    \le \binom{n-1}{2}.
\]
By \eqref{eq:square-spread-ineq2}, we obtain
\[
    |s^+(G)-s^-(G)|
    \le 2c(G)\le 2\binom{n-1}{2}
    =(n-1)(n-2).
\]

Suppose that $n\ge3$ and equality holds.  Then equality must hold in
\[
    c(G)\le\binom{n-\kappa}{2}\le\binom{n-1}{2},
\]
which forces $\kappa=1$ and $|E(G)|=\binom n2$.  Hence $G\cong K_n$.
Conversely, the adjacency eigenvalues of $K_n$ are $n-1$ and $-1$ with
multiplicity $n-1$, so
\[
    s^+(K_n)-s^-(K_n)
    =(n-1)^2-(n-1)
    =(n-1)(n-2).
\qedhere\]
\end{proof}

\subsection{Adjacency inertia}

For a graph $G$, let $n^+(G)$, $n^0(G)$, and $n^-(G)$ denote,
respectively, the numbers of positive, zero, and negative eigenvalues of its
adjacency matrix, counted with multiplicity.  Let $\iota(G)$ denote the
number of isolated vertices.

Zhang~\cite[Corollary~1.5]{Zh} proved that every graph without isolated
vertices satisfies
\[
    \min\{s^+(G),s^-(G)\}
    \ge
    \max\{n^+(G),n^0(G),n^-(G)\}.
\]
The following result proves the stronger conjecture of Elphick recorded
in~\cite[Conjecture~7.4]{Zh} and extends it to
graphs with isolated vertices.

\begin{corollary}\label{cor:inertia}
Every graph $G$ satisfies
\[
    \min\{s^+(G),s^-(G)\}
    \ge
    n^0(G)-\iota(G)+\max\{n^+(G),n^-(G)\}.
\]
\end{corollary}

\begin{proof}
Let $\eta:=\kappa(G)-\iota(G)$ denote the number of connected components of $G$ with at least two vertices.  The adjacency matrix
of any such component has at least one positive and at least one negative eigenvalue.
Hence
\[
    \min\{n^+(G),n^-(G)\}\ge \eta.
\]
Since $n=n^+(G)+n^0(G)+n^-(G)$, we obtain
\[
\begin{aligned}
    n^0(G)-\iota(G)+\max\{n^+(G),n^-(G)\}
    &=n-\iota(G)-\min\{n^+(G),n^-(G)\}\\
    &\le n-\iota(G)-\eta\\
    &=n-\kappa(G).
\end{aligned}
\]
The result follows from Corollary~\ref{cor:components}.
\end{proof}

\subsection{Positive and negative \texorpdfstring{$p$}{p}-energies}
For $p>0$, the corresponding positive and negative $p$-energies are
defined by
\[
    \mathcal E_p^+(G)
    :=\sum_{\lambda_i>0}|\lambda_i|^p,
    \qquad
    \mathcal E_p^-(G)
    :=\sum_{\lambda_i<0}|\lambda_i|^p.
\]
These quantities have also been studied in several recent
papers~\cite{AKMPp,CWZ,ETZ,LTPath,TLW}. Our main theorem resolves a number of conjectures surrounding these energies.

For connected graphs, the next result proves
\cite[Conjecture~5.4]{ETZ}.  Its negative-sign case for $p\ge2$ also proves
\cite[Conjecture~2]{AKMPp}.

\begin{corollary}\label{cor:p-energy}
Let $G$ be a graph on $n$ vertices. For every $p>0$ and $\sigma\in\{+,-\}$,
\[
    \mathcal E_p^\sigma(G)\ge
    \begin{cases}
        \left(n-\kappa(G)\right)^{p/2},&0<p\le2,\\[2mm]
        n-\kappa(G),&p\ge2.
    \end{cases}
\]
\end{corollary}

\begin{proof}
Set $N:=n-\kappa(G)$. The case $N=0$ is immediate.  Suppose that $N>0$, fix
$\sigma\in\{+,-\}$, and let $x_1,\ldots,x_r$ be the absolute values of the
adjacency eigenvalues of sign $\sigma$.  By
Corollary~\ref{cor:components},
\[
    \sum_{i=1}^r x_i^2=s^\sigma(G)\ge N.
\]
If $0<p\le2$, then
\[
    \mathcal E_p^\sigma(G)
    =\sum_{i=1}^r x_i^p
    \ge
    \left(\sum_{i=1}^r x_i^2\right)^{p/2}
    \ge N^{p/2}.
\]
Now let $p\ge2$.  By H\"older's inequality,
\[
    \mathcal E_p^\sigma(G)
    \ge
    r^{1-p/2}
    \left(\sum_{i=1}^r x_i^2\right)^{p/2}.
\]
Every nontrivial connected component of $G$ contributes an adjacency
eigenvalue of sign opposite to $\sigma$, while every isolated vertex
contributes a zero eigenvalue.  Thus at least $\kappa(G)$ eigenvalues are
not among $x_1,\ldots,x_r$, and hence $r\le n-\kappa(G)=N$. Since $1-p/2\le0$, it follows that
\[
    \mathcal E_p^\sigma(G)
    \ge
    N^{1-p/2}N^{p/2}
    =N.
\qedhere\]
\end{proof}

For a connected graph, Corollary~\ref{cor:p-energy} gives
\[
    \min\{\mathcal E_p^+(G),\mathcal E_p^-(G)\}
    \ge(n-1)^{p/2}
    \qquad (0<p\le2).
\]
This proves~\cite[Conjecture~5.4]{ETZ}; the range $0<p<1$ was previously
proved in~\cite[Theorem~5.7]{ETZ}. With the convention $\mathcal E_0^\pm(G):=n^\pm(G)$, the corresponding assertion at $p=0$ is immediate.

If $G$ is connected and $p\ge2$, then the negative-sign bound gives
\[
    \mathcal E_p^-(G)\ge n-1=\mathcal E_p^-(K_n).
\]
This proves~\cite[Conjecture~2]{AKMPp}.  The conjecture was previously known
for all real $p\ge4$ by~\cite{AKMPp} and for every integer $p\ge3$
by~\cite{CWZ}.

The bounds in Corollary~\ref{cor:p-energy} are sharp.  For $0<p<2$,
equality is attained by the disjoint union of a star and isolated vertices.
For $p=2$, equality holds for every forest.  For $p>2$, equality is attained
by a disjoint union of edges and isolated vertices.  In the connected
negative-sign case, equality is attained by $K_n$.

For connected graphs and $p>2$, Corollary~\ref{cor:p-energy} gives
\[
    \mathcal E_p^+(G)\ge n-1.
\]
The stronger conjecture
\[
    \mathcal E_p^+(G)\ge\mathcal E_p^+(P_n),
\]
where $P_n$ is the path on $n$ vertices, remains open in general.  It is
known for connected bipartite graphs for every real $p\ge2$, for all
connected graphs when $p\ge3$ is an odd integer, and for all connected
graphs when $p=4$; see~\cite{LTPath}.

\subsection{Graph complements}

Let $\overline G$ denote the complement of $G$.  Elphick and
Aouchiche~\cite[Conjecture~6]{EA} conjectured that
\[
    s^+(G)+s^+(\overline G)\le(n-1)^2
\]
and observed that this inequality would follow from
Conjecture~\ref{conj:EFGW}.  The following corollary proves their conjecture
and its negative-sign analogue.

\begin{corollary}\label{cor:nordhaus-gaddum}
Let $G$ be a graph on $n\ge2$ vertices.  For every
$\sigma\in\{+,-\}$,
\[
    s^\sigma(G)+s^\sigma(\overline G)\le(n-1)^2.
\]
For $\sigma=+$, equality holds if and only if $G$ is complete or edgeless.
For $\sigma=-$, equality holds if and only if $n=2$.
\end{corollary}

\begin{proof}
Fix $\sigma\in\{+,-\}$, and let $\tau$ denote the opposite sign.  Write $\kappa:=\kappa(G)$ and $\overline\kappa:=\kappa(\overline G)$. By Corollary~\ref{cor:components},
\[
    s^\tau(G)+s^\tau(\overline G)
    \ge 2n-\kappa-\overline\kappa.
\]
Since
\[
    s^+(G)+s^-(G)+s^+(\overline G)+s^-(\overline G)
    =n(n-1),
\]
we obtain
\[
    s^\sigma(G)+s^\sigma(\overline G)
    \le n(n-1)-2n+\kappa+\overline\kappa.
\]
At least one of $G$ and $\overline G$ is connected, and hence $\kappa+\overline\kappa\le n+1$. Therefore
\[
    s^\sigma(G)+s^\sigma(\overline G)\le(n-1)^2.
\]

Suppose that equality holds. Then $\kappa+\overline\kappa=n+1$. Since at least one of $G$ and $\overline G$ is connected, the other has
$n$ connected components.  Thus one of the two graphs is complete and the
other is edgeless.  Now
\[
    s^+(K_n)=(n-1)^2,
    \qquad
    s^-(K_n)=n-1,
\]
whereas both square energies of an edgeless graph vanish.  The equality
statements follow.
\end{proof}

\section*{Acknowledgments and AI disclosure}
We used ChatGPT to generate exploratory code for testing conjectures and to help polish the language of proof drafts. All mathematical content was verified by the authors, and the proof of Theorem~\ref{thm:main} has additionally been formally verified in the Lean proof assistant; the formalization is available at \url{https://github.com/ShengtongZhang-alt/Sq}. The core ideas and the overall strategy of the proof were developed independently by the authors; the motivation and the path leading to the argument are described in the overview; any suggestions from the tool were modified or discarded as needed, and the authors take full responsibility for the correctness and originality of the paper.

Q.~Tang thanks Minghua Lin for helpful discussions.


\begin{thebibliography}{99}

\bibitem{AbiadEtAl}
A.~Abiad, L.~de~Lima, D.~N.~Desai, K.~Guo, L.~Hogben, and J.~Madrid,
\newblock Positive and negative square energies of graphs,
\newblock \emph{Electron. J. Linear Algebra} \textbf{39} (2023), 307--326.

\bibitem{AKMP}
S.~Akbari, H.~Kumar, B.~Mohar, and S.~Pragada,
\newblock A linear lower bound for the square energy of graphs,
\newblock \emph{Electron. J. Combin.} \textbf{32} (2025), no.~3,
Paper No.~P3.53.

\bibitem{AKMPp}
S.~Akbari, H.~Kumar, B.~Mohar, and S.~Pragada,
\newblock Vertex partitioning and $p$-energy of graphs,
\newblock \emph{Linear Algebra Appl.} \textbf{724} (2025), 96--107.

\bibitem{AKMPZ}
S.~Akbari, H.~Kumar, B.~Mohar, S.~Pragada, and S.~Zhang,
\newblock Refinement of a conjecture on positive square energy of graphs,
\newblock arXiv:2506.07264, 2025.

\bibitem{AL}
T.~Ando and M.~Lin,
\newblock Proof of a conjectured lower bound on the chromatic number of a
graph,
\newblock \emph{Linear Algebra Appl.} \textbf{485} (2015), 480--484.

\bibitem{BSM}
A.~Berman and N.~Shaked-Monderer,
\newblock \emph{Completely Positive Matrices},
\newblock World Scientific Publishing Co., River Edge, NJ, 2003.

\bibitem{CWZ}
Z.~Chen, Z.~Wang, and X.-D.~Zhang,
\newblock Positive and negative $3$-energies of graphs,
\newblock arXiv:2604.15656, 2026.

\bibitem{CS}
G.~Coutinho and T.~J.~Spier,
\newblock Sums of squares of eigenvalues and the vector chromatic number,
\newblock arXiv:2308.04475, 2023.

\bibitem{CSZ}
G.~Coutinho, T.~J.~Spier, and S.~Zhang,
\newblock Conic programming to understand sums of squares of eigenvalues of
graphs,
\newblock arXiv:2411.08184, 2024.

\bibitem{EA}
C.~Elphick and M.~Aouchiche,
\newblock Nordhaus--Gaddum and other bounds for the sum of squares of the
positive eigenvalues of a graph,
\newblock \emph{Linear Algebra Appl.} \textbf{530} (2017), 150--159.

\bibitem{EFGW}
C.~Elphick, M.~Farber, F.~Goldberg, and P.~Wocjan,
\newblock Conjectured bounds for the sum of squares of positive eigenvalues
of a graph,
\newblock \emph{Discrete Math.} \textbf{339} (2016), no.~9, 2215--2223.

\bibitem{EL}
C.~Elphick and W.~Linz,
\newblock Symmetry and asymmetry between positive and negative square
energies of graphs,
\newblock \emph{Electron. J. Linear Algebra} \textbf{40} (2024), 418--432.

\bibitem{ETZ}
C.~Elphick, Q.~Tang, and S.~Zhang,
\newblock A spectral lower bound on chromatic numbers using $p$-energy,
\newblock \emph{European J. Combin.} \textbf{132}, Part B (2026),
Article No.~104252.

\bibitem{Hong}
Y.~Hong,
\newblock A bound on the spectral radius of graphs,
\newblock \emph{Linear Algebra Appl.} \textbf{108} (1988), 135--139.

\bibitem{HJ}
R.~A.~Horn and C.~R.~Johnson,
\newblock \emph{Matrix Analysis},
\newblock 2nd ed., Cambridge University Press, Cambridge, 2013.

\bibitem{LSG}
X.~Li, Y.~Shi, and I.~Gutman,
\newblock \emph{Graph Energy},
\newblock Springer, New York, 2012.

\bibitem{LiuNing2023}
L.~Liu and B.~Ning,
\newblock Unsolved problems in spectral graph theory,
\newblock \emph{Oper. Res. Trans.} \textbf{27} (2023), no.~4, 33--60.

\bibitem{LTPath}
Y.~Liu and Q.~Tang,
\newblock Path-minimality for positive $p$-energies, Laplacian-type spectra, and line graphs,
\newblock arXiv:2606.30996v1, 2026.

\bibitem{LTZTuran}
Y.~Liu, Q.~Tang, and S.~Zhang,
\newblock A positive square-energy strengthening of Tur\'an's theorem,
\newblock preprint, 2026.

\bibitem{NZ}
B.~Ning and J.~Zeng,
\newblock A proof of a conjecture on positive and negative square energies
of unicyclic graphs,
\newblock arXiv:2605.24668, 2026.

\bibitem{TLW}
Q.~Tang, Y.~Liu, and W.~Wang,
\newblock On the positive and negative $p$-energies of graphs under edge
addition,
\newblock \emph{Discrete Appl. Math.} \textbf{388} (2026), 25--33.

\bibitem{VB}
L.~Vandenberghe and S.~Boyd,
\newblock Semidefinite programming,
\newblock \emph{SIAM Rev.} \textbf{38} (1996), no.~1, 49--95.

\bibitem{WE}
P.~Wocjan and C.~Elphick,
\newblock New spectral bounds on the chromatic number encompassing all
eigenvalues of the adjacency matrix,
\newblock \emph{Electron. J. Combin.} \textbf{20} (2013), no.~3,
Paper No.~P39.

\bibitem{Zh}
S.~Zhang,
\newblock Extremal values for the square energies of graphs,
\newblock arXiv:2409.15504v2, 2024.

\end{thebibliography}
\end{document}